\documentclass[a4j,12pt]{article} 
\usepackage{amsmath}
\usepackage{amssymb}
\usepackage{amsfonts}
\usepackage{cases}
\usepackage{graphicx}
\usepackage{mathrsfs}
\usepackage{amsthm}
\usepackage{enumerate}
\usepackage[all]{xy} 
\usepackage{setspace}
\usepackage{latexsym}
\DeclareMathAlphabet{\mathpzc}{OT1}{pzc}{m}{it}
\newtheorem{defi}{Definition}[section]
\newtheorem{theo}[defi]{Theorem}
\newtheorem{prop}[defi]{Proposition}

\newtheorem{exam}[defi]{Example}

\newtheorem{conj}[defi]{Conjecture}

\newtheorem{rema}[defi]{Remark}

 \title{The affine property of quasi-free states on self-dual CAR algebras}
 \author{Yusuke SAWADA}
\date{}
\begin{document}
\maketitle
\begin{center}
\underline{Abstract}
\end{center}
We study conditions for which the correspondence between quasi-free states on self-dual CAR algebras and covariant operators is affine.

\section{Introduction}

Araki\cite{arak} generalized the classification of gauge invariant quasi-free states of canonical anticommutation relations obtained by Powers-St\o
mer\cite{powe-stom} to arbitrary quasi-free states up to quasi-equivalence. More specifically, he found a necessary and sufficient condition for which the GNS representations $\pi_\phi$ and $\pi_{\phi'}$ associated with quasi-free states $\phi$ and $\phi'$ on a self-dual CAR algebra are quasi-equivalent. 

In this paper, we provide a new problem with respect to the convexity of the correspondence between quasi-free states and covariant operators, and partially solve it. Suppose $\phi_S$ is a quasi-free state on a self-dual CAR algebra associated with a covariance operator $S$. In general, for given covariant operators $S$ and $S'$, the correspondence is not affine, that is $\phi_{\lambda
S+(1-\lambda)S'}\neq\lambda\phi_S+(1-\lambda)\phi_{S'}$ holds for some $0\leq\lambda\leq1$. Hence the set of all quasi-free states is not a convex. The problem is when the correspondence is affine, and was pointed out by Professor Shigeru Yamagami. We find a necessary condition for which the correspondence is affine, and moreover a necessary and sufficient condition for it when in presence of a finite dimensional Hilbert space and commuting covariant operators $S$ and $S'$.

\section{Preliminaries}

In this section, we recall the notion of quasi-free states on self-dual CAR algebras based on \cite{arak}.

Let $\mathcal{H}$ be a complex Hilbert space with scalar product $(\cdot,\cdot)$ and $\Gamma:\mathcal{H}\to\mathcal{H}$ an anti-unitary involution on $\mathcal{H}$. A self-dual CAR algebra $A(\mathcal{H},\Gamma)$ is defined as a universal $C^*$-algebra generated by a set $\{b(\xi)\mid\xi\in\mathcal{H}\}$ with relations
\begin{eqnarray*}
&&b(\alpha\xi+\beta\eta)=\alpha
b(\xi)+\beta
b(\eta)\quad\mbox{if}\ \alpha,\beta\in\mathbb{C},\ \xi,\eta\in\mathcal{H},\\
&&b(\Gamma\xi)=b(\xi)^*\quad\mbox{if}\ \xi\in\mathcal{H},\\
&&[b(\xi),b(\eta)]_{+}:=b(\xi)b(\eta)^*+b(\eta)^*b(\xi)=(\xi,\eta){\bf1}\quad\mbox{if}\ \xi,\eta\in\mathcal{H}.
\end{eqnarray*}
For every state $\phi$ on a self-dual CAR algebra $A(\mathcal{H},\Gamma)$, there exists a positive contractive operator $S$ on $\mathcal{H}$ such that
\begin{eqnarray}
(S\xi,\eta)=\phi(b(\eta)^*b(\xi))
\end{eqnarray}
for all $\xi,\eta\in\mathcal{H}$, and satisfying
\begin{eqnarray}
0\leq
S\leq{\bf1},\\
\Gamma
S\Gamma={\bf1}-S.
\end{eqnarray}
The operator $S$ is called a covariance operator of $\phi$. Quasi-free
states are defined as follows.
\begin{defi}
A state $\phi$ on $A(\mathcal{H},\Gamma)$ is called quasi-free if it satisfies
\begin{eqnarray*}
\phi(b(\xi_1)\cdots
b(\xi_k))= \left\{
    \begin{array}{l}
\hspace{-5pt}(-1)^{\frac{(n-1)n}{2}}\displaystyle\sum_{\sigma\in
T_{2n}}\mbox{\rm{sgn}}(\sigma)\prod_{i=1}^n\phi(b(\xi_{\sigma(i)})b(\xi_{\sigma(i+n)}))\quad\mbox{if}\ k=2n,\\
0\quad\mbox{if}\ k=2n+1,
    \end{array}
  \right.
\end{eqnarray*}
where
\[
T_{2n}=\{\sigma\in
S_{2n}\mid\sigma(1)<\cdots<\sigma(n),\ \sigma(i)<\sigma(i+n)\quad\mbox{if}\ i=1,\cdots,n\}.
\]
\end{defi}
For every bounded linear operator $S$ satisfying $(2)$ and $(3)$, there exists a unique quasi-free state $\phi_S$ satisfying (1). In other words, (1) yields a one-to-one correspondence between covariance operators and quasi-free states. Hence a bounded linear operator satisfying $(2)$ and $(3)$ is simply called a covariance operator.

\section{The affine property of quasi-free states}

In this section, we discuss the convexity of quasi-free states. For covariance operators $S$ and $S'$ and $0\leq\lambda\leq1$, $\lambda
S+(1-\lambda)S'$ is a covariance operator. We say that the couple $(S,S')$ has the affine property if $\phi_{\lambda
S+(1-\lambda)S'}=\lambda\phi_S+(1-\lambda)\phi_{S'}$ holds for all $0\leq\lambda\leq1$. If $\lambda\phi_S+(1-\lambda)\phi_{S'}$ is a quasi-free state for all $0\leq\lambda\leq1$, by the uniqueness of covariant operators, the couple $(S,S')$ has the affine property. Now, we obtain a necessary condition ensuring two covariant operators have the affine property.
\begin{theo}
If a couple $(S,S')$ of covariance operators has the affine property, then $(S\xi,\Gamma\xi)=(S'\xi,\Gamma\xi)$ holds for all $\xi\in\mathcal{H}$.
\end{theo}
\begin{proof}
We assume that $\phi_{\lambda
S+(1-\lambda)S'}=\lambda\phi_S+(1-\lambda)\phi_{S'}$ and suppose there exists $\zeta\in\mathcal{H}$ such that $(S\zeta,\Gamma\zeta)\neq(S'\zeta,\Gamma\zeta)$. Fix $\xi,\eta\in\mathcal{H}.$ For every quasi-free state $\phi$, we have 
\begin{eqnarray}
\phi(b(\xi)b(\eta)b(\zeta)b(\zeta))=\phi(b(\xi)b(\eta))\phi(b(\zeta)b(\zeta)).
\end{eqnarray}
This implies that
\begin{eqnarray}
&&\lambda\phi_S(b(\xi)b(\eta)b(\zeta)b(\zeta))+(1-\lambda)\phi_{S'}(b(\xi)b(\eta)b(\zeta)b(\zeta))\nonumber\\
&&=\lambda(S\eta,\Gamma\xi)(S\zeta,\Gamma\zeta)+(1-\lambda)(S'\eta,\Gamma\xi)(S'\zeta,\Gamma\zeta).
\end{eqnarray}
On the other hand, by (4),
\begin{eqnarray}
&&\phi_{\lambda
S+(1-\lambda)S'}(b(\xi)b(\eta)b(\zeta)b(\zeta))\nonumber\\
&&=(\lambda(S\eta,\Gamma\xi)+(1-\lambda)(S'\eta,\Gamma\xi))(\lambda(S\zeta,\Gamma\zeta)+(1-\lambda)(S'\zeta,\Gamma\zeta)).
\end{eqnarray}
If we put $\alpha=(S\zeta,\Gamma\zeta)$ and $\beta=(S'\zeta,\Gamma\zeta)$, then (5) and (6) imply
\begin{eqnarray*}
&&\alpha\lambda(S\eta,\Gamma\xi)+\beta(1-\lambda)(S'\eta,\Gamma\xi)\\
&&=(\lambda(S\eta,\Gamma\xi)+(1-\lambda)(S'\eta,\Gamma\xi))(\alpha\lambda+\beta(1-\lambda)).
\end{eqnarray*}
Therefore, we have
\begin{eqnarray*}
&&\alpha\lambda
S\eta+\beta(1-\lambda)S'\eta=\lambda(\alpha\lambda+\beta(1-\lambda))
S\eta+(1-\lambda)(\alpha\lambda+\beta(1-\lambda))S'\eta\\
&&\mbox{or, equivalently}\ (\alpha-\beta)S\eta=(\alpha-\beta)S'\eta.
\end{eqnarray*}
As $\alpha\neq\beta$, we conclude that $S=S'$, as $\eta$ is arbitrary in $\mathcal{H}$. This is a contradiction.
\end{proof}

Now, we assume that a Hilbert space $\mathcal{H}$ has a finite dimension and two covariance operators commute.

Let $\mathcal{H}$ be a Hilbert space with $\dim\mathcal{H}=2n$ for some $n\in\mathbb{N}$ and $\Gamma$ be an anti-unitary involution on $\mathcal{H}$. For a covariance operator $S$, there are an orthonormal basis $\{\varepsilon_1,\cdots,\varepsilon_n,\Gamma\varepsilon_1,\cdots,\Gamma\varepsilon_n\}$ and $\alpha_1,\cdots,\alpha_n\in[0,1]$ such that
\[
S\varepsilon_i=\alpha_i\varepsilon_i,\quad
S\Gamma\varepsilon_i=(1-\alpha_i)\Gamma\varepsilon_i\quad\mbox{if}\ i=1,\cdots,n.
\]
If we define
\begin{eqnarray*}
b(\varepsilon_i):=\underbrace{\begin{pmatrix}1&0\\0&-1\end{pmatrix}\otimes\cdots\otimes\begin{pmatrix}1&0\\0&-1\end{pmatrix}}_{i-1}\otimes\begin{pmatrix}0&0\\1&0\end{pmatrix}\otimes\underbrace{\begin{pmatrix}1&0\\0&1\end{pmatrix}\otimes\cdots\otimes\begin{pmatrix}1&0\\0&1\end{pmatrix}}_{n-i}
\end{eqnarray*}
for each $i=1,\cdots,n$, then we have
\begin{eqnarray*}
&&A(\mathcal{H},\Gamma)=M(2,\mathbb{C})^{\otimes
n},\\
&&\phi_S(b)=\mbox{\rm{Tr}}\left(\left(\bigotimes_{i=1}^n\begin{pmatrix}\alpha_i&0\\0&1-\alpha_i\end{pmatrix}\right)b\right)\quad\mbox{if}\ b\in
M(2,\mathbb{C})^{\otimes
n}.
\end{eqnarray*}
Let $S'$ be covariance operators which commutes with $S$. Then we can chose $\varepsilon_i$ such that $S$ and $S'$ are characterized by eigenvalues $\alpha_i,\alpha_i'$ as
\begin{eqnarray*}
&&S\varepsilon_i=\alpha_i\varepsilon_i,\quad
S\Gamma\varepsilon_i=(1-\alpha_i)\Gamma\varepsilon_i,\\
&&S'\varepsilon_i=\alpha_i'\varepsilon_i,\quad
S'\Gamma\varepsilon_i=(1-\alpha_i')\Gamma\varepsilon_i
\end{eqnarray*}
by the simultaneous diagonalization. Suppose $0<\lambda<1$ then we have
\begin{eqnarray*}
&&(\lambda
S+(1-\lambda)S')\varepsilon_i=(\lambda\alpha_i+(1-\lambda)\alpha_i')\varepsilon_i,\\
&&(\lambda
S+(1-\lambda)S')\Gamma\varepsilon_i=(1-(\lambda\alpha_i+(1-\lambda)\alpha_i'))\Gamma\varepsilon_i.
\end{eqnarray*}
This implies that
\begin{eqnarray}
&&\lambda\phi_S(b)+(1-\lambda)\phi_{S'}(b)\nonumber\\
&&=\lambda\mbox{\rm{Tr}}\left(\left(\bigotimes_{i=1}^n\begin{pmatrix}\alpha_i&0\\0&1-\alpha_i\end{pmatrix}\right)b\right)+(1-\lambda)\mbox{\rm{Tr}}\left(\left(\bigotimes_{i=1}^n\begin{pmatrix}\alpha_i'&0\\0&1-\alpha_i'\end{pmatrix}\right)b\right),\nonumber\\
&&\hspace{1pt}\\
&&\phi_{\lambda
S+(1-\lambda)S'}(b)=\mbox{\rm{Tr}}\left(\left(\bigotimes_{i=1}^n\begin{pmatrix}\lambda\alpha_i+(1-\lambda)\alpha_i'&0\\0&1-(\lambda\alpha_i+(1-\lambda)\alpha_i')\end{pmatrix}\right)b\right)\nonumber\\
&&\hspace{1pt}
\end{eqnarray}
for all $b\in
M(2,\mathbb{C})^{\otimes
n}$. 

Any vector $\xi\in\mathcal{H}$ has a linear form $\xi=\sum_{i=1}^nx_i\varepsilon_i+\sum_{i=1}^ny_i\Gamma\varepsilon_i$ for some $x_1,\cdots,x_n,y_1,\cdots,y_n\in\mathbb{C}$. Then we have
\begin{eqnarray}
(S\xi,\Gamma\xi)=\sum_{i=1}^n\alpha_ix_iy_i+\sum_{i=1}^n(1-\alpha_i)x_iy_i=\sum_{i=1}^nx_iy_i.
\end{eqnarray}
Thus, $(S\xi,\Gamma\xi)=(S'\xi,\Gamma\xi)$ holds for all $\xi\in\mathcal{H}$

In the case when $n=2$, we consider the affine property. We put $b=\begin{pmatrix}1&0\\0&0\end{pmatrix}$ and $b'=\begin{pmatrix}0&0\\0&1\end{pmatrix}$. Then (7), (8) and equations
\begin{eqnarray*}
&&\left(\begin{pmatrix}\alpha_1&0\\0&1-\alpha_1\end{pmatrix}\otimes\begin{pmatrix}\alpha_2&0\\0&1-\alpha_2\end{pmatrix}\right)\left(b\otimes
b'\right)=\begin{pmatrix}\alpha_1&0\\0&0\end{pmatrix}\otimes\begin{pmatrix}0&0\\0&1-\alpha_2\end{pmatrix},\\
&&\left(\begin{pmatrix}\alpha_1'&0\\0&1-\alpha_1'\end{pmatrix}\otimes\begin{pmatrix}\alpha_2'&0\\0&1-\alpha_2'\end{pmatrix}\right)\left(b\otimes
b'\right)=\begin{pmatrix}\alpha_1'&0\\0&0\end{pmatrix}\otimes\begin{pmatrix}0&0\\0&1-\alpha_2'\end{pmatrix},\\
&&\left(\bigotimes_{i=1}^2\begin{pmatrix}\lambda\alpha_i+(1-\lambda)\alpha_i'&0\\0&1-(\lambda\alpha_i+(1-\lambda)\alpha_i')\end{pmatrix}\right)(b\otimes
b')\\
&&\hspace{70pt}=\begin{pmatrix}\lambda\alpha_1+(1-\lambda)\alpha_1'&0\\0&0\end{pmatrix}\otimes\begin{pmatrix}0&0\\0&1-(\lambda\alpha_2+(1-\lambda)\alpha_2')\end{pmatrix},
\end{eqnarray*}
imply that
\begin{eqnarray*}
&&(\lambda\phi_S+(1-\lambda)\phi_{S'}-\phi_{\lambda
S+(1-\lambda)S'})(b\otimes
b')\nonumber\\
&&=\lambda(\alpha_1-\alpha_1\alpha_2)+(1-\lambda)(\alpha_1'-\alpha_1'\alpha_2')\nonumber\\
&&\hspace{10pt}-(\lambda\alpha_1+(1-\lambda)\alpha_1')(1-(\lambda\alpha_2+(1-\lambda)\alpha_2'))\nonumber\\
&&=\lambda(1-\lambda)(\alpha_1-\alpha_1')(\alpha_2-\alpha_2').
\end{eqnarray*}
If $\alpha_1=\alpha_1'$, we can show that $\lambda\phi_S+(1-\lambda)\phi_{S'}=\phi_{\lambda
S+(1-\lambda)S'}$ by computations. Now, we have obtained the following proposition.
\begin{prop}
Let $\mathcal{H}$ be the four dimensional Hilbert space, $\Gamma$ be an anti-unitary involution on $\mathcal{H}$ and $(S,S')$ be a couple of mutually commuting covariance operators on $\mathcal{H}$. Then $(S,S')$ has the affine property if and only if $(\alpha_1-\alpha_1')(\alpha_2-\alpha_2')=0$. Moreover if $\lambda\phi_S+(1-\lambda)\phi_{S'}=\phi_{\lambda
S+(1-\lambda)S'}$ holds for some $0<\lambda<1$, we have $(\alpha_1-\alpha_1')(\alpha_2-\alpha_2')=0$, and hence the couple $(S,S')$ has the affine property.
\end{prop}
The last statement of Proposition 3.2 derives from commutativity of $S$ and $S'$, similarly in Theorem 3.4 as bellow.
\begin{rema}
In the case when $\dim\mathcal{H}=2$, the equation $\lambda\phi_S(b)+(1-\lambda)\phi_{S'}=\phi_{\lambda
S+(1-\lambda)S'}$ holds for all $\alpha_1,\alpha_1'\in\mathbb{C}$ automatically. 
\end{rema}
Now, suppose $\dim\mathcal{H}=2n+1$ for some $n\in\mathbb{N}$, $\Gamma$ is an anti-unitary involution and $S$ is a covariance operator. Then there exists a unit vector $\varepsilon_0\in\mathcal{H}$, in addition to $\varepsilon_1,\cdots,\varepsilon_n$ in the case when $\dim\mathcal{H}=2n$, such that
\[
S\varepsilon_0=\frac{1}{2}\varepsilon_0,\quad\Gamma\varepsilon_0=\varepsilon_0.
\]
We have
\[
b(\varepsilon_0)^*=b(\varepsilon_0),\quad
b(\varepsilon_0)^2=\frac{1}{2}{\bf1},\quad[b(\varepsilon_0),b(\varepsilon_i)]_{+}=0,\quad\mbox{if}\ i=1,\cdots,n.
\]
Since the $C^*$-subalgebra generated by $b(\varepsilon_0)$ is isomorphic to $\mathbb{C}\oplus\mathbb{C}$, we have
\begin{eqnarray}
&&A(\mathcal{H},\Gamma)=M(2,\mathbb{C})^{\otimes
n}\otimes(\mathbb{C}\oplus\mathbb{C}),\nonumber\\
&&\phi_S(b)=\mbox{\rm{Tr}}\left(\left(\left(\bigotimes_{i=1}^n\begin{pmatrix}\alpha_i&0\\0&1-\alpha_i\end{pmatrix}\right)\otimes\begin{pmatrix}\frac{1}{2}&0\\0&\frac{1}{2}\end{pmatrix}\right)b\right)\nonumber\\
&&\hspace{150pt}\mbox{if}\ b\in
M(2,\mathbb{C})^{\otimes
n}\otimes(\mathbb{C}\oplus\mathbb{C}).
\end{eqnarray}
In this case, $(S\xi,\Gamma\xi)=(S'\xi,\Gamma\xi)$ holds for all $\xi\in\mathcal{H}$ too.

We obtain the following result with respect to a necessary and sufficient condition for which $\lambda\phi_S+(1-\lambda)\phi_{S'}=\phi_{\lambda
S+(1-\lambda)S'}$ holds in the case when a Hilbert space $\mathcal{H}$ is finite dimensional.
\begin{theo}
Let $\mathcal{H}$ be a Hilbert space with a dimension $1<k<\infty$, $\Gamma$ an anti-unitary involution on $\mathcal{H}$, and $S$ and $S'$ mutually commuting covariance operators on $\mathcal{H}$. Suppose $k=2n$ or $k=2n+1$ for some $n\in\mathbb{N}$, and $\alpha_i$ and $\alpha_i'$ are coefficients associated to $S$ and $S'$, respectively. Then $(S,S')$ has the affine property if and only if $\alpha_i=\alpha_i'$ for all $i=1,\cdots,n$ except for at most one $i_0=1,\cdots,n$. Moreover if $\lambda\phi_S+(1-\lambda)\phi_{S'}=\phi_{\lambda
S+(1-\lambda)S'}$ holds for some $0<\lambda<1$, the couple $(S,S')$ has the affine property.
\end{theo}
\begin{proof}
First, we assume that $k=2n$. By (7) and (8), we have
\begin{eqnarray*}
&&(\lambda\phi_S+(1-\lambda)\phi_{S'}-\phi_{\lambda
S+(1-\lambda)S'})\left(\bigotimes_{i=1}^n\begin{pmatrix}b_{11}^i&b_{12}^i\\b_{21}^i&b_{22}^i\end{pmatrix}\right)\nonumber\\
&&=\lambda\prod_{i=1}(\alpha_ib_{11}^i+(1-\alpha_i)b_{22}^i)+(1-\lambda)\prod_{i=1}(\alpha_i'b_{11}^i+(1-\alpha_i')b_{22}^i)\nonumber\\
&&\hspace{50pt}-\prod_{i=1}^n((\lambda\alpha_i+(1-\lambda)\alpha_i')b_{11}^i+(1-(\lambda\alpha_i+(1-\lambda)\alpha_i'))b_{22}^i)\nonumber\\
&&\hspace{1pt}
\end{eqnarray*}
for all $\begin{pmatrix}b_{11}^i&b_{12}^i\\b_{21}^i&b_{22}^i\end{pmatrix}\in
M(2,\mathbb{C})$, $i=1,\cdots,n$ and $0\leq\lambda\leq1$. Thus, if there exists $i_0=1,\cdots,n$ such that $\alpha_i=\alpha_i'$ for $i\neq
i_0$, then $\lambda\phi_S+(1-\lambda)\phi_{S'}=\phi_{\lambda
S+(1-\lambda)S'}$ for all $0\leq\lambda\leq1$. Conversely, if $\lambda\phi_S+(1-\lambda)\phi_{S'}=\phi_{\lambda
S+(1-\lambda)S'}$ for some $0<\lambda<1$, then if $i\neq
j$, 
\[
(\alpha_i-\alpha_i')(\alpha_j-\alpha_j')=0.
\]
Indeed, when $i\neq
j$, after putting $b_{11}^l=b_{22}^l=1$ for $l\neq
i,j$ we can use Proposition 3.2. Consequently the statement is proved in the case when $\mathcal{H}$ has an even dimension.

When $k=2n+1$, since
\begin{eqnarray*}
&&(\lambda\phi_S+(1-\lambda)\phi_{S'}-\phi_{\lambda
S+(1-\lambda)S'})\left(\left(\bigotimes_{i=1}^n\begin{pmatrix}b_{11}^i&b_{12}^i\\b_{21}^i&b_{22}^i\end{pmatrix}\right)\otimes\begin{pmatrix}b_{11}^0&0\\0&b_{22}^0\end{pmatrix}\right)\nonumber\\
&&=\frac{1}{2}(b_{11}^0+b_{22}^0)(\lambda\prod_{i=1}(\alpha_ib_{11}^i+(1-\alpha_i)b_{22}^i)+(1-\lambda)\prod_{i=1}(\alpha_i'b_{11}^i+(1-\alpha_i')b_{22}^i)\nonumber\\
&&\hspace{75pt}-\prod_{i=1}^n((\lambda\alpha_i+(1-\lambda)\alpha_i')b_{11}^i+(1-(\lambda\alpha_i+(1-\lambda)\alpha_i'))b_{22}^i))\\
&&\hspace{100pt}\mbox{if}\ \begin{pmatrix}b_{11}^i&b_{12}^i\\b_{21}^i&b_{22}^i\end{pmatrix}\in
M(2,\mathbb{C}),\ i=1,\cdots,n,\ b_{11}^0,b_{22}^0\in\mathbb{C}
\end{eqnarray*}
from (10), the thesis follows the case $k=2n$.
\end{proof}
\begin{rema}
If we change the wording of Theorem 3.4, since covariance operators correspond to quasi-free states affinely in the case when $\dim\mathcal{H}=2$, see Remark 3.3, the correspondence is affine when the difference of two covariance operators has rank two.
\end{rema}

In the finite dimensional case, CAR algebras and quasi-free states have concrete representations, and hence we could analyze the affine property for two commuting covariance operators.
\begin{rema}
For a couple $(S,S')$ of covariance operators, the condition $(S\xi,\Gamma\xi)=(S'\xi,\Gamma\xi)$ for all $\xi\in\mathcal{H}$ in Theorem 3.1 is not a sufficient condition for that the couple has the affine property. In fact in the setting of Theorem 3.4, by (9) the condition holds. However, by Theorem 3.4, we can get an example $(S,S')$ which does not have the affine property.
\end{rema}
Now we discuss a relation between the affine property and the dimension of the image of $S-S'$ for covariant operators $S$ and $S'$, i.e. the rank of $S-S'$.
\begin{prop}
If $S$ and $S'$ are covariance operators on a Hilbert space $\mathcal{H}$ with dimension $k$ such that  $\dim\ker(S-S')\geq
k-1$, then $S=S'$.
\end{prop}
\begin{proof}
Assume $\dim\ker(S-S')=k-1$. In other words, the eigenspace of $S-S'$ for $0$ has the dimension $k-1$. On the other hand, by $\Gamma(S-S')=-(S-S')\Gamma$, if $\lambda\neq0$ is an eigenvalue of $S-S'$, then $-\lambda$ is so. Hence the direct sum of the eigenspaces of $S-S'$ for all non-zero eigenvalues has an even dimension. This is a contradiction.

We have proved that $\dim\ker(S-S')=k$, i.e. $S=S'$.
\end{proof}
By Theorem 3.4, for every couple $(S,S')$ of covariant commutative operators on the $k$ dimensional Hilbert space $\mathcal{H}$ has the affine property, we have $\dim{\rm
Im}(S-S')\leq2$. 

Now we consider the affine property for two non-commutative covariance operators in the finite dimensional case.
\begin{exam}
Suppose $\mathcal{H}=\mathbb{C}^3$, $\Gamma:\mathcal{H}\to\mathcal{H}$ is a unitary involution defined by
\[
\Gamma\begin{pmatrix}x\\y\\z\end{pmatrix}=\begin{pmatrix}\overline{y}\\
\overline{x}\\
\overline{z}\end{pmatrix}
\]
for each $\begin{pmatrix}x\\y\\z\end{pmatrix}\in\mathcal{H}$ and
\[
S=\frac{1}{6}\begin{pmatrix}2&0&1\\0&4&-1\\1&-1&3\end{pmatrix},\quad
S'=\frac{1}{6}\begin{pmatrix}3&0&1\\0&3&-1\\1&-1&3\end{pmatrix}.
\]
Then $S, S'$ are non-commutative covariance operators and $\dim{\rm
Im}(S-S')=2$. The eigenvalues of $S$ and $S'$ are $\frac{1}{2},\frac{3+\sqrt{3}}{6},\frac{3-\sqrt{3}}{6}$ and  $\frac{1}{2},\frac{3+\sqrt{2}}{6},\frac{3-\sqrt{2}}{6}$, respectively. If we put $\lambda=\frac{1}{2}$ and $b=\begin{pmatrix}-1&0\\0&1\end{pmatrix}\otimes\begin{pmatrix}1&0\\0&1\end{pmatrix}\in
M(2,\mathbb{C})\otimes(\mathbb{C}\oplus\mathbb{C})$, then we have
\begin{eqnarray*}
\phi_{S}(b)=-\frac{\sqrt{3}}{3},\quad\phi_{S'}(b)=-\frac{\sqrt{2}}{3},\quad\phi_{\lambda
S+(1-\lambda)S'}(b)=-\frac{1}{2}
\end{eqnarray*}
by (10), and hence the couple $(S,S')$ does not have the affine property.
\end{exam}
As opposed to the commutative case, every couple $(S,S')$ with $\dim{\rm
Im}(S-S')=2$ always does not have the affine property as Example 1. Now, we can make naturally the following conjecture.
\begin{conj}
Let $(S,S')$ be a couple of covariant operators having the affine property. If $S\neq
S'$, then $\dim{\rm
Im}(S-S')=2$.
\end{conj}
\begin{rema}
Finally, we consider the affine property for quasi-free states on a self-dual CCR algebra treated in {\rm\cite{arak-shir}} and {\rm\cite{arak2}}. In this case, a quasi-free state is defined as in Definition 2.1 without signature. Moreover, there also exists a one-to-one correspondence between quasi-free states and hermitian forms satisfying some properties. By similar computations as those used in the proof of Theorem 3.1, we can prove that the correspondence is affine if and only if two hermitian forms coincide.
\end{rema}

\section*{Acknowledgment}
I am grateful to Professor Shigeru Yamagami for helpful comments.

\ 

\ 

\ 

\ 

\ 

\ 

\ 

\ 

\ 

\noindent
Yusuke Sawada\\
Graduate School of Mathematics,\\
Nagoya University,\\
Furocho, Chikusa-ku, Nagoya, 464-8602, Japan.\\
e-mail: m14017c@math.nagoya-u.ac.jp
\end{document}